\documentclass{amsart}

\usepackage{amsmath}
\usepackage{amsthm}
\usepackage{amssymb,amsfonts,stmaryrd,paralist,tikz}
\usepackage[mathscr]{euscript}
\usetikzlibrary{matrix,arrows,decorations.pathmorphing}

\newcommand{\R}{\mathbb{R}}

\newcommand{\C}{\mathbb{C}}

\newcommand{\CP}{\mathbb{C}\mathrm{P}}

\newcommand{\two}{I\!\!I}
\newcommand{\three}{I\!\!I\!\!I}
\newcommand{\tr}{\mathrm{tr}}

\newtheorem{thm}{Theorem}[section]
\newtheorem*{thm*}{Theorem}
\newtheorem{lem}[thm]{Lemma}
\newtheorem*{lem*}{Lemma}
\newtheorem{cor}[thm]{Corollary}
\newtheorem*{cor*}{Corollary}
\newtheorem{prop}[thm]{Proposition}
\newtheorem*{prop*}{Proposition}

\newtheorem*{defn*}{Definition}

\newtheorem*{question*}{Question}

\newtheorem{bigthm}{Theorem}

\numberwithin{equation}{section}

\usepackage{color}
\definecolor{verydarkblue}{rgb}{0,0,0.4}
\usepackage{hyperref}
\hypersetup{
pdfauthor={Keaton Quinn},
pdftitle={Conformally Flat Structures},
colorlinks=true,linkcolor=verydarkblue,
citecolor=verydarkblue,urlcolor=verydarkblue
}
\renewcommand{\H}{\mathbb{H}}

\begin{document}

\title{Conformally flat structures via hyperbolic geometry}
\author{Keaton Quinn}
\address{Department of Mathematics, Boston College, Chestnut Hill, MA} 
\date{Date: March 29, 2026}

\begin{abstract}
A pair of tensors $(g,B)$ form the induced metric and shape operator of an immersion into hyperbolic space if and only if they satisfy the Gauss-Codazzi equations. 
Such a pair of tensors induce a pair $(\hat{g},\hat{B})$ related to the ideal boundary of hyperbolic space.
Krasnov and Schlenker, and Bridgeman and Bromberg show in the surface case that there is a duality between $(g,B)$ and $(\hat{g},\hat{B})$. 
Moreover, $(g,B)$ solves the Gauss-Codazzi equations if and only if $(\hat{g},\hat{B})$ solve a corresponding set of equations. 
We show a similar duality exists and identify these corresponding equations for an arbitrary dimension, as well as show there exists a unique solution for $\hat{B}$ provided  $\hat{g}$ is locally conformally flat. 
As an application, we offer a proof of the Weyl-Schouten theorem concerning locally conformally flat metrics that factors through hyperbolic geometry. 
\end{abstract}

\maketitle

\section{Introduction}

Given an immersed surface $\Sigma$ in hyperbolic space with induced metric $g$ and shape operator $B$, the Gauss-Codazzi equations for $(g,B)$ relate the intrinsic curvature of $\Sigma$ to the geometry of hyperbolic space.
The equations are
\[
\left\{
\begin{aligned}
Rm(g) &= -\frac{1}{2}g \owedge g + \frac{1}{2}\two \owedge \two \\
d^\nabla B &= 0,
\end{aligned}
\right.
\]
where $\two(X,Y) = g(BX,Y)$ and the Kulkarni-Nomizu product $h \owedge k$ is defined in the next section.
Krasnov and Schlenker in \cite{Krasnov-Schlenker2008} showed that, under certain curvature restrictions, the image of $\Sigma$ under the hyperbolic Gauss map to the sphere induces a pair of tensors $(\hat{g},\hat{B})$ on $\Sigma$ that satisfy a dual set of equations which they call the \emph{Gauss-Codazzi equations at infinity}. 
The algebraic relation $(g,B) \to (\hat{g},\hat{B})$ is invertible and, independently, Bridgeman-Bromberg in \cite{Bridgeman-Bromberg2022} and Schlenker in \cite{Schlenker2017} showed that $(g,B)$ solves the Gauss-Codazzi equations if and only if $(\hat{g},\hat{B})$ solves this dual set of equations. 

Here we generalize this work to an arbitrary dimension.
We show that a pair of tensors $(g,B)$ on an $n$-dimensional manifold $M$ that satisfy the Gauss-Codazzi equations induce, via the same transformations \eqref{TensorsAtInfinity} and \eqref{Tensors} as in 2-dimensions, a pair $(\hat{g},\hat{B})$ on $M$ that solve the dual set of equations
\[
\left\{
\begin{aligned}
Rm(\hat{g}) &= -\frac{1}{2}\hat{g} \owedge \hat{\two} \\
d^{\widehat{\nabla}}\hat{B} &= 0.
\end{aligned}
\right.
\]
When $n=2$ these equations reduce to those of Krasnov and Schlenker from \cite{Krasnov-Schlenker2008}. 
We similarly show that $(g,B)$ solves the Gauss-Codazzi equations if and only if $(\hat{g},\hat{B})$ solves the Gauss-Codazzi equations at infinity (Theorem \ref{DualEquations}).

Osgood and Stowe in \cite{Osgood-Stowe1992} introduced a traceless tensor $\mathrm{OS}(g_2,g_1)$ built from a pair of conformal metrics $g_2 = e^{2u}g_1$.
In \cite{Bridgeman-Bromberg2022}, Bridgeman and Bromberg show that, when $n = 2$, a choice of complex projective structure on $\Sigma$ compatible with $\hat{g}$ produces a solution to the Gauss-Codazzi equations at infinity in the form of $\hat{\two} = 2\mathrm{OS}(\hat{g}) - K(\hat{g})\hat{g}$, where $\mathrm{OS}(\hat{g})$ is the Osgood-Stowe tensor of $\hat{g}$ relative to the flat metric of any projective chart, and where $K(g)$ is the Gaussian curvature of a metric $g$.

The situation is slightly different when $n > 2$.
In our setting, due to Liouville's theorem for conformal maps, the notion of a complex projective structure generalizes to that of a locally conformally flat structure. 
For our purposes, a metric $g$ is locally conformally flat if there is an atlas of charts to $\R^n$ on which $g$ is conformal to the Euclidean metrics of the charts.
We prove the following.

\begin{bigthm}
\label{bigthm1}
A pair $(\hat{g},\hat{B})$ solves the Gauss-Codazzi equations at infinity if and only if the metric $\hat{g}$ is locally conformally flat and
\[
\hat{\two} = 2\mathrm{OS}(\hat{g}) - \frac{S(\hat{g})}{n(n-1)}\hat{g},
\]
where $S$ is the scalar curvature and where $\mathrm{OS}(\hat{g})$ is the Osgood-Stowe tensor of $\hat{g}$ with respect to the induced locally conformally flat structure.
\end{bigthm}
\noindent
As $S = 2K$ when $n=2$, this theorem generalizes the solutions found by Bridgeman and Bromberg to any dimenion.

As an application of our work, we use this duality to prove a classical theorem concerning the equivalence of a metric being locally conformally flat with the vanishing of certain tensors constructed from the metric.
As such, this work provides an alternative proof of this Schouten-Weyl theorem from a hyperbolic-geometric perspective.

\begin{bigthm}[Schouten-Weyl]
Let $M$ be a manifold of dimension $n \geq 3$.
A metric $g$ on $M$ is locally conformally flat if and only if its Schouten tensor is Codazzi ($n = 3$) or its Weyl tensor vanishes ($n \geq 4$).
\end{bigthm}
\noindent
Our proof relates the Schouten and Weyl tensors to the Gauss-Codazzi equations at infinity and the novelty is that we are able to prove all cases together at once. 

\subsection*{Acknowledgements}
We thank Martin Bridgeman for our extensive conversations on this topic as well as for his helpful feedback on drafts of this paper.
Thanks go to Franco Vargas Pallete for helpful discussions and suggestions, and to Jos\'e M. Espinar for connecting us with a body of work that investigates similar questions. We also thank Clifford Taubes for pointing out an error in an earlier version of Lemma 2.2.

\section{Preliminaries}

\subsection{Tensors}

Let $M$ be a smooth manifold carrying a Riemannian metric $g$. 
The non-degeneracy of $g$ at each point allows us to identify $TM$ with $T^*M$ via the map that sends $v \in T_pM$ to $g_p(v,\cdot) \in T^*_pM$ and this extends to an identification of vector fields and 1-forms. 
Abusing notation, we write $g: TM \to T^*M$ for this map and $g^{-1}:T^*M \to TM$ for its inverse.
If $B:TM \to TM$ is a $(1,1)$-tensor, self-adjoint with respect to $g$, then $gB$ is a symmetric 2-tensor defined by
\[
gB(X,Y) = g(BX,Y),
\]
for all vector fields $X$ and $Y$.
Similarly, if $T$ is a symmetric 2-tensor, then the endomorphism field $g^{-1}T$ defined by 
\[
T(X,Y) = g( (g^{-1}T)X,Y)
\]
is self-adjoint.
The trace of such a $2$-tensor is the trace of the corresponding $(1,1)$-tensor
\[
\tr_g(T) = \tr(g^{-1}T).
\]

The Levi-Civita connection $\nabla$ for $g$ is the unique torsion free connection on $TM$ compatible with $g$.
The exterior covariant derivative $d^\nabla$ is the alternization of this connection.
For $B$ an endomorphism, $d^\nabla B$ is defined by
\begin{align*}
d^\nabla B(X,Y) &= (\nabla_XB)Y - (\nabla_YB)X \\
&= \nabla_X(BY)- \nabla_Y(BX) - B([X,Y]),
\end{align*}
for vector fields $X$ and $Y$.
For a symmetric 2-tensor the derivative is 
\begin{align*}
d^\nabla T(X,Y,Z)
&= (\nabla_YT)(X,Z)- (\nabla_Z T)(X,Y) \\
&= YT(X,Z) - T(\nabla_YX,Z) - T(X,\nabla_YZ) \\
&\phantom{=} - ZT(X,Y) + T(\nabla_ZX,Y) + T(X,\nabla_ZY).
\end{align*}
These two are related via $d^\nabla T(X,Y,Z) = g(X, d^\nabla(g^{-1}T)(Y,Z))$ so that one vanishes if and only if the other does.
The curvature of $\nabla$ is 
\[
R^\nabla(X,Y)Z = \nabla_X\nabla_Y Z - \nabla_Y\nabla_X Z - \nabla_{[X,Y]}Z
\]
and the full curvature tensor 
\[
Rm(X,Y,Z,W) = g(R^\nabla(X,Y)Z,W).
\]
The curvature tensor has several important symmetries.
Given two symmetric 2-tensors $T$ and $S$, there is an operation that produces a 4-tensor with these same symmetries called the Kulkarni-Nomizu product of $T$ and $S$. 
It is defined as
\begin{align*}
T \owedge S(X,Y,Z,W)
&= T(X,W)S(Y,Z) + T(Y,Z)S(X,W) \\
&\phantom{=} - T(X,Z)S(Y,W) - T(Y,W)S(X,Z)
\end{align*}
and is a symmetric and bilinear product on symmetric 2-tensors.

Now suppose $g$ is a Riemannian metric.
Given a 4-tensor $Q$, we can produce a 2-tensor by taking the trace, with respect to $g$, of $Q$ on its first and last indices
\[
\tr_g(Q)(X,Y) = \tr_g( Q(\cdot \, , X,Y, \cdot \,)).
\]
If $T$ is a symmetric 2-tensor, the trace of the Kulkarni-Nomizu product $g \owedge T$ is computed in terms of $T$ and its trace
\begin{equation}
\label{TraceKN}
\tr_g( g \owedge T) = (n-2)T + \tr_g(T)g.
\end{equation}
The Ricci curvature tensor of $g$ is the (symmetric) 2-tensor defined by taking the trace of the curvature tensor $Ric(g) = \tr_g(Rm)$.
The scalar curvature $S(g)$ is the trace of its Ricci tensor.
The sectional curvature $sec$ of the metric is a function on 2-planes in the tangent bundle. 
If $X$ and $Y$ are a basis for such a 2-plane, then
\[
sec(X,Y) = \frac{2Rm(X,Y,Y,X)}{g \owedge g(X,Y,Y,X)},
\]
and $g$ will have constant sectional curvature $K$ if and only if its curvature tensor is of the form 
\[
Rm = \frac{1}{2}K g \owedge g.
\]

\subsection{Tensor Transformations}

Here we list the ways several geometric quantities behave under certain transformations of the metric.
See, for example, \cite{Lee2018} as a general reference.

It is well known how curvatures behave under a conformal change of metric.
We record their behavior in the following lemma.

\begin{lem}
\label{ConfChange}
Let $\tilde{g} = e^{2u}g$ be a metric conformal to a given metric $g$.
Call $\widetilde{\nabla}$ and $\nabla$ the corresponding compatible connections.
We have 
\[
\widetilde{\nabla}_XY = \nabla_XY + du(X)Y + du(Y)X - g(X,Y)\nabla u,
\]
for $\nabla u$ the gradient of $u$ with respect to $g$.
The full curvature tensors are then related by
\[
Rm(\tilde{g}) = e^{2u}Rm(g) - \tilde{g}\owedge(\mathrm{Hess}_g(u) - du^2 + \frac{1}{2}|\nabla u|^2g),
\]
and the scalar curvatures obey
\[
S(\tilde{g}) = e^{-2u}(S(g) -2(n-1)\Delta u  - (n-2)(n-1)|\nabla u|^2).
\]
\end{lem}

We will need to know the following transformations under the pullback of a metric by a tensor field, see \cite{Bridgeman-Bromberg2022} for details.

\begin{lem}
\label{TensorChange}
Let $g$ be a Riemannian metric on $M$ and $A$ an endomorphism field of $TM$.
Define the tensor $\hat{g}(X,Y) = g(AX,AY)$.
Then $\hat{g}$ is a metric provided $\lambda = 0$ is not an eigenvalue of $A$. 
In this case,
\[
\widehat{\nabla}_XY = A^{-1}(\nabla_X(AY)),
\]
is a connection compatible with $\hat{g}$. If in addition $d^{\nabla}A = 0$, then this connection is torsion free, so that it is the Levi-Civita connection of $\hat{g}$.
Moreover, the curvature tensors of these two metrics are related by 
\[
Rm(\hat{g})(X,Y,Z,W) = Rm(g)(X,Y,AZ,AW).
\]
\end{lem}

We will need to know how the exterior covariant derivative behaves under a conformal change to the metric.

\begin{lem}
\label{dRel}
Suppose $g$ is a Riemannian metric and $\tilde{g} = e^{2u}g$ is a metric conformal to $g$. Let $\widetilde{\nabla}$ and $\nabla$ be the corresponding Levi-Civita connections. 
If $T$ is a symmetric 2-tensor then the exterior covariant derivatives are related by 
\[
d^{\widetilde{\nabla}} T(X,Y,Z) = d^\nabla T(X,Y,Z) + T \owedge g (\nabla u ,X,Y,Z)
\]
for $\nabla u$ the gradient of $u$ with respect to $g$.
\end{lem}
\begin{proof}
We have
\begin{align*}
d^{\widetilde{\nabla}}T(X,Y,Z)
&= YT(X,Z) - T(\widetilde{\nabla}_YX,Z) - T(X, \widetilde{\nabla}_YZ) \\
&\phantom{=} - ZT(X,Y) + T(\widetilde{\nabla}_ZX,Y) + T(X,\widetilde{\nabla}_ZY).
\end{align*}
To simplify, we use how the connections for $\tilde{g}$ and $g$ are related.
Using Lemma \ref{ConfChange} on each relevant term in the derivative leads to a good amount of cancellation, eventually resulting in 
\begin{align*}
d^{\widetilde{\nabla}}T(X,Y,Z)
&=YT(X,Z) - T(\nabla_YX,Z) - T(X, \nabla_YZ) \\
&\phantom{=} - ZT(X,Y) + T(\nabla_ZX,Y) + T(X,\nabla_ZY) \\
&\phantom{=} + T(\nabla u ,Z)g(X,Y) + T(X,Y)du(Z) \\
&\phantom{=} - T(\nabla u, Y)g(X,Z) - T(X,Z)du(Y).
\end{align*}
The first set of the terms form $d^\nabla T(X,Y,Z)$ and after writing $du = g(\nabla u , \cdot)$, we notice the second set of the terms to be $T\owedge g(\nabla u ,X,Y,Z)$.
\end{proof}

\begin{lem}
\label{Hess}
Let $g$ be a Riemannian metric with Levi-Civita connection $\nabla$ and let $u$ be a smooth function, then
\[
d^\nabla \mathrm{Hess}_g(u) (X,Y,Z) = Rm(\nabla u, X, Y, Z)
\]
\end{lem}

\begin{proof}
To ease notation, we write $\mathrm{Hess}$ for $\mathrm{Hess}_g(u)$ when no confusion is likely.
The derivative is
\[
d^\nabla\mathrm{Hess}_g(u)(X,Y,Z) = (\nabla_Y\mathrm{Hess})(X,Z) - (\nabla_Z \mathrm{Hess})(X,Y)
\]
and $\nabla_Y\mathrm{Hess}$ can be computed via
\begin{align*}
(\nabla_Y\mathrm{Hess})(X,Z)
&= \nabla_Y(\mathrm{Hess}(X,Z)) - \mathrm{Hess}(\nabla_YX,Z) - \mathrm{Hess}(X,\nabla_YZ) \\
&= Yg(\nabla_Z(\nabla u),X) - g(\nabla_Z(\nabla u ), \nabla_YX) - g(\nabla_{\nabla_YZ}(\nabla u),X) \\
&= g(\nabla_Y \nabla_Z (\nabla u ), X) + g(\nabla_Z(\nabla u), \nabla_YX) \\
&\phantom{=} \hspace{1cm}  - g(\nabla_Z(\nabla u ), \nabla_YX) - g(\nabla_{\nabla_YZ}(\nabla u),X) \\
&= g( (\nabla_Y \nabla_Z  - \nabla_{\nabla_YZ})\nabla u, X).
\end{align*}
Similarly, 
\[
(\nabla_Z \mathrm{Hess})(X,Y) = g( (\nabla_Z \nabla_Y  - \nabla_{\nabla_ZY})\nabla u, X).
\]
The exterior covariant derivative is then 
\begin{align*}
d^\nabla\mathrm{Hess}_g(u)(X,Y,Z)
&= g( \nabla_Y \nabla_Z \nabla u - \nabla_Z \nabla_Y \nabla u - \nabla_{\nabla_YZ - \nabla_ZY} \nabla u , X) \\
&= g( R^\nabla(Y,Z)\nabla u ,X),
\end{align*}
where we have used that the connection is torsion free so that $\nabla_YZ - \nabla_ZY = [Y,Z]$.
The lemma then follows from symmetries of the curvature tensor. 
\end{proof}

\subsection{Tangent Bundles and Submanifolds}

The tangent bundle $\pi:TM \to M$ of a smooth manifold records differential information and so second order differential information naturally lives in the double tangent bundle $T(TM) \to TM$.
A connection $\nabla$ on $TM$ lets us split the double tangent bundle into a horizontal and vertical bundle $T(TM) \simeq H \oplus V$.
Here the vertical bundle $V$ is canonically defined as $V = \ker d\pi$ while the horizontal bundle $H$ depends on the connection.
It consists of all tangent vectors to curves $(\gamma(t),X(t))$ in $TM$ for which $D_tX = 0$, where $D_t = \gamma^*\nabla$ is the covariant derivative along the curve $\gamma$.

In fact, with respect to the splitting, if $\alpha(t) = (\gamma(t),X(t))$ is a path in $TM$ for a vector field $X$ along $\gamma$ then $\alpha'(t) = (\gamma'(t),D_tX) \in H \oplus V$.
Now suppose $f:\Sigma \to M$ is a smooth map and that a lift $F: \Sigma \to TM$ of $f$ is given by $F(x) = (f(x),V(x))$ for some vector field $V$ along $f$, i.e., for some $V \in \Gamma(f^*TM)$.
Then the derivative of $F$ is $dF_x(u) = (df_x(u), (f^*\nabla)_u V)$ for $f^*\nabla$ the pullback connection on $f^*TM$.

Now suppose $M$ carries a Riemannian metric with compatible connection $\widetilde{\nabla}$ and let $f: \Sigma \to M$ be an isometric immersion of a smooth manifold with metric $g$ into $M$.
The normal bundle of $\Sigma$ in $M$ is $T^\perp\Sigma$, the orthogonal complement of $T\Sigma$ inside the pullback bundle $f^*TM$.
That is, we have a splitting
\[
f^*TM \simeq T\Sigma \oplus T^\perp\Sigma.
\]
A smooth nonzero section of $T^\perp\Sigma$ is called a normal vector field.
If $X$ and $Y$ are vector fields on $\Sigma$, then $df(Y)$ is a vector field along $f$, and the pullback connection $(f^*\widetilde{\nabla})_Xdf(Y)$ decomposes under this splitting as the sum of a piece tangent to $\Sigma$ and a piece orthogonal. 
The second fundamental form $\alpha$ of the immersion is the portion which is orthogonal, and the Gauss formula relates the tangential piece to the induced Levi-Civita connection $\nabla$ on $\Sigma$.
For vector fields $X$ and $Y$, we have
\[
(f^*\widetilde{\nabla})_Xdf(Y) = df(\nabla_XY) + \alpha(X,Y).
\]
If $\Sigma$ is a hypersurface in $M$ then $T^\perp\Sigma$ is spanned by any unit normal vector field $N$ and $\alpha$ can be written as $\alpha(X,Y) = \two(X,Y)N$ for a symmetric 2-tensor $\two$ on $\Sigma$.
In this case we also refer to $\two$ as the second fundamental form of the immersion. 
The shape operator is defined to be the endomorphism field $B = g^{-1}\two$.

Given the immersion $f$, the map $F: \Sigma \to TM$ by $F = (f,N)$ we call a unit normal lift of $f$.
The Weingarten equation says $(f^*\widetilde{\nabla})_uN = -df(Bu)$ so that, from above, 
\[
dF_x(u) = (df_x(u),-df_x(Bu)).
\]

\subsection{Hyperbolic Space}
We now summarize the theory of families of parallel surfaces in hyperbolic space $\H^n$.
The following exposition can also be found in \cite{Krasnov-Schlenker2008},\cite{Bonini-Espinar-Qing2015},\cite{Abanto-Espinar2019},\cite{Bridgeman-Bromberg2022} from a variety of viewpoints.

The hyperboloid model of hyperbolic space is an embedded hypersurface of Minkowski space for which the second fundamental form is known. 

\begin{lem}
\label{HypAsSub}
Let $\H^n \to \R^{n,1}$ be the hyperboloid model of hyperbolic space as an embedded Riemannian submanifold of Mikowski space. Call $\bar{\nabla}$ the Levi-Civita connection of Minkowski space and $\nabla^{\H}$ that of hyperbolic space. Then the Gauss formula of this embedding is 
\[
\bar{\nabla}_XY = \nabla^{\H}_XY + \left<X,Y\right>N
\]
for $N(p) = p$ the unit normal vector field of the hyperboloid and $\left<X,Y \right>$ the Minkowski metric.
\end{lem}

The geodesic flow on the unit tangent bundle $\mathcal{G}^t :U \H^n \to \H^n$ (projected down) has a nice expression in this model. 
For $p$ a point in hyperbolic space and $v$ a unit vector tangent at $p$, the geodesic flow is
\[
\mathcal{G}^t(p,v) = \cosh(t)p + \sinh(t)v.
\]
The derivative of the geodesic flow is a map $d(\mathcal{G}^t) : T(U\H^n) \to T\H^n$.
We can compute it as follows.

\begin{lem}
\label{DerGeo}
Let $\mathcal{G}^t: U\H^n \to \H^n$ be the geodesic flow projected down to hyperbolic space. With respect to the splitting of the double tangent bundle defined above, for $(x,y)$ tangent to $U\H^n$ at $(p,v)$, the derivative of the geodesic flow is 
\[
d(\mathcal{G}^t)_{(p,v)}(x,y) = \cosh(t)x + \sinh(t)y + \sinh(t)\left<x,v\right>p.
\]
\end{lem}

\begin{proof}
Let $\alpha(s) = (\gamma(s),V(s))$ be a curve in $U\H^n$ with $\alpha(0) = (p,v)$ and $\alpha'(0) = (x,y)$. 
Recall this means that $\gamma'(0) = x$ and $D_sV(0) = y$.
We compute
\begin{align*}
d(\mathcal{G}^t)_{(p,v)}(x,y)
&= \left. \frac{d}{ds} \mathcal{G}^t(\gamma(s),V(s)) \right|_{s=0} \\
&= \left. \frac{d}{ds} \cosh(t)\gamma(s) + \sinh(t)V(s) \right|_{s=0} \\
&= \cosh(t)\gamma'(0) + \sinh(t)V'(0)
\end{align*}
where $V'$ is the derivative of $V: (-1,1) \to \R^{n,1}$, which is equal to $\bar{D}_sV(0)$.
By the Gauss equation for hyperbolic space (Lemma \ref{HypAsSub}), this derivative is 
\begin{align*}
\bar{D}_sV(0) 
&= D_sV(0) + \left<\gamma'(0),V(0)\right>\gamma(0) \\
&= y + \left< x, v \right>p,
\end{align*}
and this gives the result. 
\end{proof}

The hyperbolic Gauss map $\mathcal{G}:U\H^n \to S^{n-1}$ is the map that sends $(p,v)$ to the ideal endpoint of the geodesic staring at $p$ traveling in the direction $v$. 
One of multiple equivalent ways to interpret this is to take the geodesics in the hyperboloid model $\H^n$, send them to the Poincar\'e model $\mathbb{B}^n \subset \R^n \subset \R^{n,1}$ with a stereographic projection $\pi: \H^n \to \mathbb{B}^n$ and take a limit as $t\to \infty$ in the Euclidean topology. 
This map $\pi$ has the form
\[
\pi(p) = \frac{1}{1 - \left< p, e_{n+1} \right>}(p + \left<p,e_{n+1}\right> e_{n+1}),
\]
see \cite{Lee2018} for details. 
Doing this, one gets 
\begin{align*}
\mathcal{G}(p,v) 
&= \lim_{t \to \infty} \pi(\mathcal{G}^t(p,v)) \\
&= -\frac{1}{\left<p+v,e_{n+1}\right>} (p+v) - e_{n+1},
\end{align*} 
which belongs to the sphere $S^{n-1} \subset \R^n \subset \R^{n,1}$ since $p+v$ lies in the light cone. 

The derivative of the hyperbolic Gauss map may be computed similarly to the derivative of the geodesic flow. 

\begin{lem}
Let $\mathcal{G}:U\H^n \to S^{n-1}$ be the hyperbolic Gauss map on the hyperboloid model given by 
\[
\mathcal{G}(p,v) = -\frac{1}{\left<p+v,e_{n+1}\right>}(p+v) - e_{n+1},
\]
then the derivative is 
\[
d\mathcal{G}_{(p,v)}(x,y)
=\frac{\left< x + y + \left<x,v\right>p, e_{n+1} \right>}{\left< p + v, e_{n+1} \right>^2}(p+v) - \frac{1}{\left< p + v, e_{n+1} \right>}(x + y + \left<x,v\right>p)
\]
with respect to the splitting $T(U\H^n) \simeq H \oplus V$.
\end{lem}

\subsection{Parallel Surfaces}

Let $M$ be a smooth $n$-dimensional manifold and let $f: M \to \H^{n+1}$ be an immersion for which $M$ has a smooth unit normal vector field $N$.
A parallel family of hypersurfaces is obtained by flowing $M$ in its normal direction. 
\\
\\
\noindent
{\bf Convention:} Following the literature, we define the shape operator $B: TM \to TM$ as usual by $g^{-1}\two$ relative to the normal vector field $N$, but we flow in the \emph{opposite} direction $-N$. 
\\
\\
Given $f$ and $N$ one can form $F = (f ,-N)$ the unit normal lift $M \to U\H^{n+1}$.
If one then follows this with the geodesic flow (projected down to $\H^{n+1}$) then we get the desired parallel surfaces.

The following expansion of the induced metric on the parallel hypersurface is well known, see e.g., \cite{Krasnov-Schlenker2007}.
For completeness and clarity, we include a proof tailored to our setting. 

\begin{lem}
Let $f^t: M \to \H^{n+1}$ be defined by $f^t = \mathcal{G}^t \circ F$, then $f^t$ is an immersion provided none of the eigenvalues of $B$ are equal to $-\coth(t)$ and the induced metric is given by 
\[
g_t(X,Y) = g( (\cosh(t)Id + \sinh(t)B)X, (\cosh(t)Id + \sinh(t)B)Y)
\]
\end{lem}

\begin{proof}
The derivative of $f^t$ may be computed via the chain rule $df^t_x = d\mathcal{G}^t_{F(x)} \circ dF_x$.
Since $F: M \to U\H^{n+1}$, the derivative at $x$ takes values in $T_{F(x)}U\H^{n+1}$.
The connection splits this into the horizontal and vertical space, and via this decomposition, the derivative is 
\[
dF_x(v) = (df_x(v),\nabla_v(-N)) = (df_x(v),df_x(Bv))
\]
(recall the convention to still use $B$ relative to $+N$).
The derivative of the geodesic flow is given by Lemma \ref{DerGeo} as
\[
d\mathcal{G}^t_{(p,v)}(x,y) = \cosh(t)x + \sinh(t)y + \sinh(t)\left< x,v \right> p.
\]
Since $df_x(v)$ and $N(x)$ are orthogonal, we get
\begin{align*}
df^t_x(v) &= d\mathcal{G}^t_{F(x)}(dF_x(v)) \\
&= \cosh(t)df_x(v) + \sinh(t)df_x(Bv) \\
&= df_x(\cosh(t)Id + \sinh(t)B)v.
\end{align*}

This identifies the pullback tensor as $g_t = (f^t)^*g_{\H^{n+1}} = A_t^*g$, for $A_t = \cosh(t)Id + \sinh(t)B$.
This is a metric (equivalently $f^t$ is an immersion) whenever it is positive definite.
That is, whenever the eigenvalues of $A_t^2$ are all positive, which reduces to the condition that the eigenvalues of $A_t$ be nonzero. 
In terms of the eigenvalues $\lambda_i$ for $B$, the requirement is that $\cosh(t) + \sinh(t)\lambda_i \neq 0$, or that $\lambda_i \neq -\coth(t)$ for all $i$. 
This gives the result.
\end{proof}

The following consequence was noted by Epstein in \cite{Epstein1984}.

\begin{cor}
If the eigenvalues of $B$ are in $[-1,1]$ then $f_t$ is an immersion for all $t$ and $M$ may be flown in the forwards and backwards direction for all time. 
\end{cor}

\subsection{Tensors at Infinity}

Expanding the hyperbolic trig functions in terms of exponentials gives the induced metric on the parallel surface $M_t$ as 
\[
g_t = \frac{1}{4}e^{2t}(g + 2\two + \three) + \frac{1}{2}(g - \three) + \frac{1}{4}e^{-2t}(g - 2\two + \three),
\]
where $\three(X,Y) = g(BX,BY)$.
Define $\hat{g} = g + 2\two + \three$.
The induced metrics are asymptotic to $\hat{g}$ in the sense that their conformal classes converge to that of $\hat{g}$ as $t \to \infty$, since $4e^{-2t}g_t \to g + 2\two + \three$.
Following \cite{Krasnov-Schlenker2008} we refer to $\hat{g}$ as the \emph{metric at infinity}.
This name is mostly accurate in the following sense. 
If $\mathcal{G}: U\H^{n+1} \to S^n$ is the hyperbolic Gauss map $\mathcal{G} = \lim_{t \to \infty} \mathcal{G}^t$, then the composition $f_\infty = \mathcal{G} \circ F$ gives a
surface in $S^n = \partial_\infty(\H^{n+1})$ to which the surfaces $M_t$ can be shown to limit (in a Euclidean topology). 
If we pull back the round metric $\overset{\circ}{g}$ on the sphere by $f_\infty$ then we obtain a metric on $M$ that is conformal to $\hat{g}$.
Recall our convention to flow in the $-N$ direction.

\begin{lem}
\label{HypGaussMap}
Let $f: M \to \H^{n+1}$ be an immersion with a smooth unit normal vector field $N$. If $\hat{g}$ and $f_\infty$ are defined as above then 
\[
f^*_\infty \overset{\circ}{g} = \frac{1}{\langle f - N, e_{n+2}\rangle^2} \ \hat{g}
\]
for $\overset{\circ}{g}$ the round metric on the sphere.
In particular, if $\lambda =-1$ is not an eigenvalue of $B$ then $f_\infty$ is an immersion and $\hat{g}$ and $f_\infty^*\overset{\circ}{g}$ are metrics which are conformal.
\end{lem}

\begin{proof}
The hyperbolic Gauss map in the hyperboloid model is given by 
\[
\mathcal{G}(p,v) = \frac{-1}{\left< p + v, e_{n+2} \right>} (p + v) - e_{n+2}
\]
which has derivative, using the splitting of $U\H^{n+1}$
\[
d\mathcal{G}_{(p,v)}(x,y) = 
\frac{\left< x + y + \left<x,v\right>p, e_{n+2} \right>}{\left< p + v, e_{n+2} \right>^2}(p+v) - \frac{1}{\left< p + v, e_{n+2} \right>}(x + y + \left<x,v\right>p).
\]
This can be used to compute the pullback tensor $\mathcal{G}^*\overset{\circ}{g}$ as
\[
\mathcal{G}^*\overset{\circ}{g}_{(p,v)}((x,y),(u,w))
= \frac{1}{\left< p+v, e_{n+2}\right>^2}\left( \left<x+y,u+w\right> - \left<x,v\right>\left<u,v\right>\right),
\]
where we have used that $y$ and $w$ are tangent to $U_p\H^{n+1}$ at $v$, meaning $\left<y,v\right> = \left<w,v\right> =0$.
Recalling that the derivative of $F$ is $dF_x(v) = (df_x(v),df_x(Bv))$ and that $df_x$ is orthogonal to $N$ gives 
\begin{align*}
(f_\infty)^*\overset{\circ}{g} = F^*(\mathcal{G}^*)\overset{\circ}{g}
&=\frac{1}{\left< f - N , e_{n+2}\right>^2} \left< df(Id + B),df(Id + B)\right> \\
&= \frac{1}{\left< f - N , e_{n+2}\right>^2} \hat{g},
\end{align*}
as claimed.
If $\lambda = -1$ is not an eigenvalue of $B$, then $\hat{g}$ is a metric and from the equality $(f_\infty)^*\overset{\circ}{g}$ is as well. In particular, if $\lambda =-1$ is not an eigenvalue of $B$ then $f_\infty$ is an immersion.
\end{proof}

Looking again at the expansion of $g_t$ one makes the following definitions:
\[
\hat{\two} = g - \three, \quad \hat{B} = \hat{g}^{-1}\hat{\two}, \quad \hat{\three} = g - 2\two + \three.
\]
In particular,
\begin{equation}
\label{TensorsAtInfinity}
\begin{aligned}
\hat{g} &= g(Id+B, Id + B) \\
\hat{B} &= (Id + B)^{-1}(Id-B)
\end{aligned}
\end{equation}
and these algebraic equalities can be inverted. Namely,
\begin{equation}
\label{Tensors}
\begin{aligned} 
g &= \frac{1}{4}\hat{g}(Id + \hat{B}, Id + \hat{B}) \\
B &= (Id + \hat{B})^{-1}(Id - \hat{B}).
\end{aligned}
\end{equation}
Moreover, $(Id + \hat{B})(Id + B) = 2Id$.

If tensors $(g,B)$ obtained via \eqref{Tensors} from $(\hat{g},\hat{B})$ happen to be the induced metric and shape operator of some immersion in to $\H^{n+1}$ then $(\hat{g},\hat{B})$ will be the metric and shape operator at infinty.

\section{The Gauss-Codazzi Equations}
We now extend to arbitrary dimensions the results of \cite{Schlenker2017} and \cite{Bridgeman-Bromberg2022} on the duality between tensors on immersed hypersurfaces in hyperbolic space and tensors induced by the conformal structure at infinity. 

There is a system of equations that serve as the integrability conditions for a pair of tensors $(g,B)$ to be the induced metric and shape operator of an immersion $M \to \H^{n+1}$.
Given an immersion $f: M \to \H^{n+1}$ with smooth unit normal vector field $N$ along $f$ one has the induced metric $g = f^*g_{\H^{n+1}}$ and second fundamental form $\two$ defined by the Gauss formula
\[
(f^*\nabla^{\H})_Xdf(Y) = df(\nabla_XY) + \two(X,Y)N,
\]
where $\nabla^\H$ and $\nabla$ are the corresponding Levi-Civita connections of $g_{\H^{n+1}}$ and $g$, and where $f^*\nabla^\H$ is the pullback connection on $f^*(T\H^{n+1}) \to M$.
From $\two$, one has the shape operator $B = g^{-1}\two$ and sees that $(g,B)$ obey the Gauss-Codazzi equations
\begin{equation}
\label{GC} \tag{GC}
\left\{
\begin{aligned}
Rm &= -\frac{1}{2}g \owedge g + \frac{1}{2}\two \owedge \two \\
d^\nabla B &= 0.
\end{aligned}
\right.
\end{equation}
Because hyperbolic space has constant sectional curvature $-1$, these equations may be written as
\begin{align*}
sec(X,Y) &= -1 + \two(X,X)\two(Y,Y) - \two(X,Y)^2 \\
d^\nabla B &= 0.
\end{align*}
When $n = 2$ we recover the familiar formulas for surfaces in hyperbolic space. 
\begin{align*}
K(g) &= -1 + \det(B) \\
d^\nabla B &= 0. 
\end{align*}

The Gauss-Codazzi equations are the integrability equations for a pair of tensors to be induced by an immersion.

\begin{thm}
Let $g$ be a Riemannian metric on a simply connected manifold $M$ and let $B$ be a self adjoint endomorphism of $TM$. 
Then $(g,B)$ solve the Gauss-Codazzi equations if and only if  there exists an immersion $f: M \to \H^{n+1}$ such that $g$ is the induced metric of the immersion and $B$ is its shape operator. 
In addition, the immersion $f$ is unique up to post composition with an isometry of hyperbolic space. 
\end{thm}

If $(g,B)$ solve the Gauss-Codazzi equations then we can isometrically immerse $M$ into hyperbolic space and, via parallel flowing $M$, we get the tensors at infinity $(\hat{g},\hat{B})$.
So, given two tensors $(g,B)$ solving the Gauss-Codazzi equations we obtain another pair of tensors $(\hat{g},\hat{B})$ and these two solve their own set of equations, which, following \cite{Krasnov-Schlenker2007}, we call the Gauss-Codazzi equations at infinity.
In the general $n$-dimensional setting they are 
\begin{equation}
\label{GCInf} \tag{$\widehat{\text{GC}}$}
\left\{
\begin{aligned}
\widehat{Rm} &= -\frac{1}{2} \hat{g} \owedge \hat{\two} \\
d^{\widehat{\nabla}} \hat{B} &= 0.
\end{aligned}
\right.
\end{equation}

The next lemma shows, by taking $n=2$ and $S = 2K$, that these equations reduce to those from \cite{Krasnov-Schlenker2008} and \cite{Bridgeman-Bromberg2022} in the surface case.

\begin{lem}
\label{TraceScalarCurv}
Suppose $(\hat{g},\hat{B})$ solve the Gauss equations at infinity, then 
\[
\tr(\hat{B}) = -\frac{1}{n-1}S(\hat{g}).
\]
for $S(g)$ the scalar curvature of a metric $g$.
\end{lem}

\begin{proof}
Take a trace of both sides of $\widehat{Rm} = -(1/2)\hat{g}\owedge\hat{\two}$ and simplify using \eqref{TraceKN}.
We get
\[
\widehat{Ric} = -\frac{1}{2}(n-2)\hat{\two} -\frac{1}{2}\tr_{\hat{g}}(\hat{\two})\hat{g}.
\]
Taking another trace yields
\[
S(\hat{g}) = -\frac{1}{2}(n-2)\tr_{\hat{g}}(\hat{\two}) - \frac{n}{2}\tr_{\hat{g}}(\hat{\two}),
\]
and simplifying gives the lemma.
\end{proof}

This dual set of equations is equivalent to the Gauss-Codazzi equations in the following sense.

\begin{thm}
\label{DualEquations}
Let the tensors $(g,B)$ and $(\hat{g},\hat{B})$ on $M$ be related by the algebraic identities \eqref{TensorsAtInfinity} and \eqref{Tensors}. 
Then, provided $\lambda =-1$ is not an eigenvalue of $B$ or $\hat{B}$, we have that $(g,B)$ solves the Gauss-Codazzi equations if and only if $(\hat{g},\hat{B})$ solves the Gauss-Codazzi equations at infinity.
\end{thm}

\begin{proof}
That $(\hat{g},\hat{B})$ solves the Codazzi equation at infinity if and only if $(g,B)$ solves the Codazzi equation follows from the relationship between the connections of $g$ and $\hat{g}$ (Lemma \ref{ConfChange}). 
For vector fields $X$ and $Y$,
\[
\widehat{\nabla}_XY = (Id+B)^{-1}\nabla_X(Id+B)Y.
\]
A computation then shows that, because the connections are torsion-free,
\[
d^{\widehat{\nabla}}\hat{B} = -(Id+B)^{-1}d^\nabla B.
\]
So, $d^{\widehat{\nabla}}\hat{B} = 0$ if and only if $d^\nabla B = 0$, since $\lambda = -1$ is not an eigenvalue of either $B$ or $\hat{B}$.

Now suppose $(g,B)$ solves the Gauss equation. 
We will show $(\hat{g},\hat{B})$ solves the Gauss equation at infinity. 
Recall the identities $g = \frac{1}{4}(\hat{g} + 2 \hat{\two} + \hat{\three})$ and $\two = \frac{1}{4}(\hat{g} - \hat{\three})$.
After making the substitution, the Gauss equation becomes 
\[
Rm = -\frac{1}{32}(\hat{g} + 2 \hat{\two} + \hat{\three})\owedge(\hat{g} + 2 \hat{\two} + \hat{\three}) + \frac{1}{32}(\hat{g} - \hat{\three})\owedge(\hat{g} - \hat{\three}),
\]
and expanding and some cancellation reduces this to
\[
32Rm = -4\hat{g}\owedge \hat{\two} -4 \hat{g} \owedge \hat{\three} - 4 \hat{\two}\owedge\hat{\two} - 4 \hat{\two}\owedge\hat{\three}.
\] 
Simplifying a little more produces
\[
-8Rm = (\hat{g} + \hat{\two})\owedge (\hat{\two} + \hat{\three}).
\]
Since, for example, $\hat{g}(X,Y) + \hat{\two}(X,Y) = \hat{g}((Id + \hat{B})X,Y)$, we abuse notation and can write that
\[
-8Rm = \hat{g}(Id, Id + \hat{B})\owedge \hat{g}(\hat{B}, Id + \hat{B}).
\]

Now, because $\hat{g} = g(Id + B, Id + B)$, Lemma \ref{TensorChange} shows the curvature tensors are related by $\widehat{Rm}(X,Y,Z,W) = Rm(X,Y,(Id+B)Z,(Id+B)W)$, meaning we are ultimately interested in computing
\[
\hat{g}(Id, Id + \hat{B})\owedge \hat{g}(\hat{B}, Id + \hat{B})(X,Y,(Id+B)Z,(Id+B)W).
\]
By virtue of $(Id + \hat{B})(Id+B) = 2 Id$, the first term of this product is 
\begin{align*}
\hat{g}(X,(Id + \hat{B})(Id+B&)W)\hat{g}(\hat{B}Y,(Id + \hat{B})(Id+B)Z) \\
&= 4\hat{g}(X,W)\hat{g}(\hat{B}Y,Z) \\
&= 4 \hat{g}(X,W)\hat{\two}(Y,Z).
\end{align*}
The other terms simplify similarly to give
\begin{align*}
\hat{g}(Id, Id + \hat{B})\owedge &\hat{g}(\hat{B}, Id + \hat{B})(X,Y,(Id+B)Z,(Id+B)W) \\
&= 4 \hat{g}(X,W)\hat{\two}(Y,Z) + 4 \hat{g}(Y,Z)\hat{\two}(X,W) \\
&\phantom{=} -4 \hat{g}(X,Z)\hat{\two}(Y,W) - 4 \hat{g}(Y,W)\hat{\two}(X,Z),
\end{align*}
which is $4\hat{g}\owedge\hat{\two}(X,Y,Z,W)$.
All together, we have
\[
-2Rm(X,Y,(Id+B)Z,(Id+B)W) = \hat{g}\owedge\hat{\two}(X,Y,Z,W),
\]
and this is equivalent to the Gauss equation at infinity for $(\hat{g},\hat{B})$.

Now suppose $(\hat{g},\hat{B})$ solves the Gauss equation at infinity, we will show $(g,B)$ solves the Gauss equation.
Again we abuse notation to write $\hat{g} = g(Id + B, Id + B)$ and $\hat{\two} = g(Id - B, Id + B)$, in which case the Gauss equation at infinity can be written as
\[
-2\widehat{Rm} = g(Id+B,Id+B)\owedge g(Id-B,Id+B).
\]
For similar reasons as above, the curvature tensors of $g$ and $\hat{g}$ are related by $Rm(X,Y,Z,W) = \widehat{Rm}(X,Y,\frac{1}{2}(Id+\hat{B})Z,\frac{1}{2}(Id+\hat{B})W)$.
Therefore, we want to compute
\[
g(Id+B,Id+B)\owedge g(Id-B,Id+B)(X,Y,\frac{1}{2}(Id+\hat{B})Z,\frac{1}{2}(Id+\hat{B})W).
\]
Because $(Id+\hat{B})(Id+B) = 2Id$, the first term in the product is 
\begin{align*}
g((Id+B)X,(Id + B)\frac{1}{2}(Id+\hat{B})W)&g((Id-B)Y,(Id+B)\frac{1}{2}(Id+\hat{B})Z) \\
&= g((Id+B)X,W)g((Id-B)Y,Z).
\end{align*}
The other terms in the product simplify similarly and we get that $\hat{g}\owedge\hat{\two}$ applied to $(X,Y,\frac{1}{2}(Id+\hat{B})Z,\frac{1}{2}(Id+\hat{B})W)$ reduces to 
\begin{align*}
g((&Id+B)X,W)g((Id-B)Y,Z) + g((Id+B)Y,Z)g((Id-B)X,W) \\
&-g((Id+B)X,Z)g((Id-B)Y,W) - g((Id+B)Y,W)g((Id-B)X,Z),
\end{align*}
and this is equal to $g(Id+B,Id)\owedge g(Id-B,Id)(X,Y,Z,W)$.
Consequently, the curvature tensor of $g$ is
\begin{align*}
Rm &= -\frac{1}{2}g(Id+B,Id)\owedge g(Id-B,Id) \\
&= -\frac{1}{2}(g+\two)\owedge(g-\two) \\
&= -\frac{1}{2}g \owedge g + \frac{1}{2}\two \owedge \two,
\end{align*}
which is the Gauss equation for $(g,B)$.

\end{proof}

Given a pair $(\hat{g},\hat{B})$ that solve the Gauss-Codazzi equations at infinity, transforming to $(g,B)$ gives a pair that solve the regular Gauss-Codazzi equations and guarantees an isometric immersion $f: (M,g) \to \H^{n+1}$ into hyperbolic space. 
Epstein in \cite{Epstein1984} gives a construction of this map $f$ as the envelope of a family of horospheres tangent to the image of $f_\infty$ in $S^n$.
See also \cite{Anderson1998} for a detailed discussion for the surface case.

\subsection{Conformally Flat and M\"obius Structures}

In the 2 dimensional setting, the work of Bridgeman-Bromberg \cite{Bridgeman-Bromberg2022} and Schlenker \cite{Schlenker2017} deal heavily with Riemann surfaces equipped with complex projective structures.
Such a structure is an atlas of charts to $\CP^1 \simeq S^2$ with transition functions the restrictions of M\"obius transformations in $\mathrm{PSL}_2\C$.
A Riemann surface by contrast is a surface with an atlas of charts to $\C$ whose transition functions are holomorphic. 
As M\"obius transformations are holomorphic, a complex projective structure induces a complex structure on a surface.
But in general the concepts are distinct. 
In fact, the set of complex projective structures on a surface (up to equivalence) fibers over the Teichm\"uller space of the surface, with fiber the vector space of holomorphic quadratic differentials.
See \cite{Dumas2009} for a survey of these topics. 

For dimension $n \geq 3$, a complex projective structure generalizes to that of a M\"obius structure on a manifold.
A M\"obius transformation in arbitrary dimension $n$ is a diffeomorphism of $S^n$ which is the composition of inversions in spheres. 
A M\"obius structure on $M$ then is an atlas of charts to $S^n$ whose transition functions are (the restrictions of) M\"obius transformations.
This is a geometric structure with topological space $S^n$ and group $\text{M\"ob}(S^n)$. 
As such, any M\"obius structure on a manifold $M$ can be given via a developing map $f: \tilde{M} \to S^n$ on its universal cover and a holonomy representation $\rho: \pi_1(M) \to \text{M\"ob}(S^n)$ satisfying $f(\gamma \cdot x) = \rho(\gamma)f(x)$ for all $x \in \tilde{M}$ and $\gamma \in \pi_1(M)$.

A complex structure on a surface generalizes to a locally conformally flat structure.
Indeed, a locally conformally flat structure on a manifold $M$ is an atlas of charts to $\R^n$ whose transition functions are conformal maps of the Euclidean metric.
There is an equivalent definition in terms of Riemannian metrics.
We say a metric $g$ is locally conformally flat if each point in $M$ has a chart to $\R^n$ on which $g$ is conformal to the pullback of the Euclidean metric.
A locally conformally flat structure on $M$ is also the conformal class $[g]$ of a locally conformally flat metric $g$.
See \cite{Matsumoto1992} for this equivalence as well as survey of these topics.

In contrast to the surface case, these two notions are one and the same.
This is because conformal transformations in dimensions $n \geq 3$ are incredibly rigid in the following sense (see \cite{Matsumoto1992} for a modern proof).

\begin{thm}[Liouville]
Suppose $U$ is a domain in $\R^n$ for $n \geq 3$, and $\varphi: U \to \R^n$ is a conformal map for the Euclidean metric. Then $\varphi$ is a M\"obius transformation.
\end{thm}

\begin{cor}
Any (locally) conformally flat structure is a M\"obius structure. 
\end{cor}
\begin{proof}
The transition functions of a conformally flat structure are conformal maps and hence M\"obius transformations.
\end{proof}

Given two conformal metrics $g_2 = e^{2u}g_1$, Osgood and Stowe define in \cite{Osgood-Stowe1992} the symmetric 2-tensor $\mathrm{OS}(g_2,g_1)$ as the traceless part of $\mathrm{Hess}(u) - du^2$.
That is, 
\[
\mathrm{OS}(g_2,g_1) = \mathrm{Hess}(u)  - du\otimes du - \frac{1}{n}\left( \Delta u - |\nabla u|^2\right)g_1,
\]
with all relevant objects defined with respect to $g_1$.
We will refer to this as the Osgood-Stowe differential of $g_2$ with respect to $g_1$, or sometimes simply as the Osgood-Stowe tensor.
They prove that $\mathrm{OS}$ has a cocycle property for three conformal metrics
\[
\mathrm{OS}(g_3,g_1) = \mathrm{OS}(g_3,g_2) + \mathrm{OS}(g_2,g_1).
\]
There is also a naturality property that if $f: M \to M$ is a smooth map then for two conformal metrics
\[
f^*\mathrm{OS}(g_2,g_1) = \mathrm{OS}(f^*g_2,f^*g_1).
\]
It also is sensitive to M\"obius transformations in the following sense. 
If $f:U \to \R^n$ is (the restriction of) a M\"obius transformation on a domain in $\R^n$ then for $\bar{g}$ the Euclidean metric, we have $\mathrm{OS}(f^*\bar{g},\bar{g}) = 0$. 

These facts have the following consequence. 
If $g$ is a locally conformally flat metric on $M$ and $\varphi: U \to \R^n$ is a chart for which $g = e^{2u}\varphi^*\bar{g}$, then $\mathrm{OS}(g, \varphi^*\bar{g})$ patches together to form a global tensor on $M$. 
To see this, if $\psi: V \to \R^n$ is any other conformal chart overlapping with $\varphi$ then 
\begin{align*}
\mathrm{OS}(g ,\psi^*\bar{g}) 
&= \mathrm{OS}(g, \varphi^*\bar{g}) + \mathrm{OS}(\varphi^*\bar{g},\psi^*\bar{g}) \\
&= \mathrm{OS}(g, \varphi^*\bar{g}) + \varphi^*\mathrm{OS}(\bar{g},(\psi \circ \varphi^{-1})^*\bar{g})
\end{align*}
and $\mathrm{OS}(\bar{g},(\psi \circ \varphi^{-1})^*\bar{g}) = 0$ since $\psi \circ \varphi^{-1}$ is a conformal map and hence a M\"obius transformation.
For a conformal metric $g$ on $M$ we will refer to this globally defined object as $\mathrm{OS}(g)$.

\section{Results}

\begin{lem}
\label{scalings}
If $(\hat{g},\hat{B})$ are a pair of tensors that solve the Gauss-Codazzi equations at infinity, then so too do $(e^{2t}\hat{g},e^{-2t}\hat{B})$ for any $t$. 
\end{lem}

\begin{proof}
Call $\hat{g}_t = e^{2t}\hat{g}$ and $\hat{B}_t = e^{-2t}\hat{B}$.
Then the second fundamental form $\hat{\two}_t$ is constant in $t$ since $\hat{\two}_t(X,Y) = e^{2t}\hat{g}(e^{-2t}\hat{B}X,Y) = \hat{\two}(X,Y)$.
Consequently, 
\[
-\frac{1}{2}\hat{g}_t \owedge \hat{\two}_t 
= -\frac{1}{2}e^{2t}\hat{g}\owedge \hat{\two} 
= e^{2t}Rm(\hat{g}).
\]
The lemma is then proven by noting that if $k$ is a positive constant and $g$ a Riemannian metric, then $Rm(kg) = kRm(g)$.
\end{proof}

If $(\hat{g},\hat{B})$ solves the Gauss-Codazzi equations at infinity, then transforming to $(g,B)$ provides an immersion $f: M \to \H^{n+1}$ into hyperbolic space.
If we do the same for $(e^{2t}\hat{g},e^{-2t}\hat{B})$ we obtain a hypersurface in $f_t: M \to \H^{n+1}$ which may be obtained from $f$ by flowing in the normal direction for time $t$.
That is $f_t  = \mathcal{G}^t \circ F$ for $F$ the unit normal lift of $f$ and for $\mathcal{G}^t$ the geodesic flow. 
In this way, a pair $(\hat{g},\hat{B})$ solving \ref{GCInf} produces not just an immersion into hyperbolic space, but a whole family of immersions.

\begin{prop}
\label{GCInf-LCF}
If $(\hat{g},\hat{B})$ are a pair of tensors that solve the Gauss-Codazzi equations at infinity, then $\hat{g}$ is locally conformally flat. 
\end{prop}

\begin{proof}
First suppose $\lambda = -1$ is not an eigenvalue of $\hat{B}$.
Let $p$ be a point in $M$ and take a simply connected neighborhood $U$ around $p$.
On $U$, define $(g,B)$ via \eqref{Tensors}.
Since $-1$ is not an eigenvalue of $\hat{B}$, the tensor $g$ is a Riemannian metric and since $(\hat{g},\hat{B})$ solve \ref{GCInf}, the pair $(g,B)$ will solve \ref{GC}.
We are then gifted an isometric immersion $f: U \to \H^{n+1}$ and composing this with the hyperbolic Gauss map produces $f_\infty:U \to S^n$, from which Lemma \ref{HypGaussMap} says $\hat{g}$ is conformal to $f_\infty^*\overset{\circ}{g}$.
Since the round metric $\overset{\circ}{g}$ on the sphere is locally conformally flat, by shrinking $U$ if necessary, $f_\infty^*\overset{\circ}{g}$ will be conformal to a flat metric, implying $\left.\hat{g}\right|_{U}$ will be as well. 

In the case that an eigenvalue of $\hat{B}$ is $-1$ at a point $p$, take a simply connected neighborhood $U$ of $p$ with compact closure.
There is a constant $c$ with $-\infty < c < -1$ such that, on $\bar{U}$, the eigenvalues of $\hat{B}$ are uniformly bounded below by $c$. 
Choose $t$ large enough so that $e^{-2t}c > -1$. 
From Lemma \ref{scalings}, for this same $t$, the pair $(e^{2t}\hat{g},e^{-2t}\hat{B})$ also solve \ref{GCInf}. 
Moreover, on $\bar{U}$, none of the eigenvalues of $e^{-2t}\hat{B}$ are $-1$. 
This is because if $\lambda$ is an eigenvalue of $e^{-2t}\hat{B}$ then $e^{2t}\lambda$ is an eigenvalue of $\hat{B}$ , meaning $e^{2t}\lambda > c$.
Hence $\lambda > e^{-2t}c > -1$.
Consequently, from the preceding reasoning, $e^{2t}\hat{g}$ is locally conformally flat on $U$.
Shrink $U$ if necessary so that $e^{2t}\hat{g}$ is conformal to a flat metric on $U$. 
Then $\hat{g}$ is conformal to a flat metric on $U$ as well since $\hat{g}$ is conformal to $e^{2t}\hat{g}$.
\end{proof}

When $\hat{g}$ is locally conformally flat, the Gauss-Codazzi equations at infinity are a system of equations for the tensor $\hat{B}$.
We are fully able to identify the solutions.
We note that Theorems \ref{MainThmGauss} and \ref{MainThmCodazzi} and their proofs are still valid when $n = 2$ once a complex projective structure has been chosen to create the tensor $\mathrm{OS}(\hat{g})$.

\begin{thm}
\label{MainThmGauss}
Let $\hat{g}$ be a locally conformally flat metric on $M$. Then for 
\[
\hat{g} \hat{B} = 2\mathrm{OS}(\hat{g}) - \frac{S(\hat{g})}{n(n-1)}\hat{g},
\]
$(\hat{g},\hat{B})$ solves the Gauss equation at infinity.
\end{thm}

\begin{proof}
This is a quick consequence of how the curvatures behave under a conformal change.
Locally write $\hat{g} = e^{2u}\bar{g}$ for a flat metric $\bar{g}$. Then by Lemma \ref{ConfChange},
\[
S(\hat{g}) = e^{-2u}(S(\bar{g}) -2(n-1)\Delta u  - (n-2)(n-1)|\bar{\nabla}u|^2)
\]
and
\begin{equation}
\label{ConfCurv}
\widehat{Rm} = e^{2u}\bar{Rm} - e^{2u}\bar{g}\owedge(\mathrm{Hess}_{\bar{g}}(u) - du^2 + \frac{1}{2}|\nabla u|^2\bar{g}).
\end{equation}
Since $\bar{g}$ is flat, we know $S(\bar{g}) = 0$ and $\bar{Rm} = 0$.
From this, a computation shows that
\[
-\frac{S(\hat{g})}{n(n-1)} = \left( \frac{2}{n}\Delta u + \left(1 - \frac{2}{n}\right)|\bar{\nabla}u|^2 \right)\bar{g}.
\]
The full second fundamental form at infinity can now be simplified to 
\begin{align}
2\mathrm{OS}(\hat{g},\bar{g}) &- \frac{S(\hat{g})}{n(n-1)}\hat{g} \nonumber \\
&= 2\mathrm{Hess}_{\bar{g}}(u) - 2du^2 - \frac{2}{n}\left(\Delta u - |\bar{\nabla}u|^2\right)\bar{g} + \left( \frac{2}{n}\Delta u + \left(1 - \frac{2}{n}\right)|\bar{\nabla}u|^2 \right)\bar{g} \nonumber \\
&= 2\mathrm{Hess}_{\bar{g}}(u) - 2 du^2 + |\bar{\nabla}u|^2 \bar{g}, \label{OSu}
\end{align}
and then substituting this into \eqref{ConfCurv} shows that $\widehat{Rm} = -\frac{1}{2}\hat{g}\owedge \hat{\two}$.
Consequently, $\hat{B}$ defined in the theorem satisfies the Gauss equation at infinity.
\end{proof}

\begin{thm}
\label{MainThmCodazzi}
Let $\hat{g}$ be a locally conformally flat metric on $M$. Then for 
\[
\hat{g} \hat{B} = 2\mathrm{OS}(\hat{g}) - \frac{S(\hat{g})}{n(n-1)}\hat{g},
\]
$(\hat{g},\hat{B})$ solves the Codazzi equation at infinity.
\end{thm}

\begin{proof}
Because $d^{\widehat{\nabla}}(g^{-1}T) = 0$ if and only if $d^{\widehat{\nabla}}T = 0$ for any symmetric 2-tensor $T$, it suffices to compute the derivative of $2\mathrm{OS} + \frac{1}{n-n^2}\hat{S}\hat{g}$ and show it vanishes.
And, since terms in the Osgood-Stowe tensor are (locally) computed relative to a flat metric $\bar{g}$, we can use Lemma \ref{dRel} to instead compute things in terms of $d^{\bar{\nabla}}$.
In fact, the computation becomes more manageable once we simplify our second fundamental form at infinity so that every thing is in terms of $\bar{g}$.
And this has already been done in \eqref{OSu}.

Here is the idea of the computation:
From Lemma \ref{dRel} we can write
\begin{equation}
\label{ToComp}
d^{\widehat{\nabla}}(- 2 du^2 + |\bar{\nabla}u|^2 \bar{g}) = d^{\bar{\nabla}}(- 2 du^2 + |\bar{\nabla}u|^2 \bar{g}) + (- 2 du^2 + |\bar{\nabla}u|^2 \bar{g})\owedge \bar{g}.
\end{equation}
The first term on the right hand side will be the negative of $d^{\widehat{\nabla}}(2\mathrm{Hess}_{\bar{g}}(u))$ while the the second term has the good sense to vanish. 
This means $d^{\widehat{\nabla}}(- 2 du^2 + |\bar{\nabla}u|^2 \bar{g}) = - d^{\widehat{\nabla}}(2\mathrm{Hess}_{\bar{g}}(u))$, implying the second fundamental form satisfies the Codazzi equation at infinity.

We again write $\mathrm{Hess}$ for $\mathrm{Hess}_{\bar{g}}(u)$ to ease notation. 
From Lemmas \ref{dRel} and \ref{Hess}, 
\begin{equation}
\label{SimpDerHess}
\begin{aligned}
d^{\widehat{\nabla}}2\mathrm{Hess}(X,Y,Z)
&= d^{\bar{\nabla}}2\mathrm{Hess}(X,Y,Z) + 2\mathrm{Hess}\owedge \bar{g}(\bar{\nabla}u,X,Y,Z) \\
&= 2\bar{Rm}(\bar{\nabla}u,X,Y,Z) + 2\mathrm{Hess}\owedge \bar{g}(\bar{\nabla}u,X,Y,Z), \\
&= 2\mathrm{Hess}\owedge \bar{g}(\bar{\nabla},X,Y,Z)
\end{aligned}
\end{equation}
since $\bar{Rm} = 0$ because $\bar{g}$ is flat. 
As for $d^{\bar{\nabla}}(- 2 du^2 + |\bar{\nabla}u|^2 \bar{g})$, we compute the derivative of each term separately.
Using that $du^2(U,V) = du(U)du(V)$ for any two vector fields $U$ and $V$, we have that the first term here is 
\begin{align*}
d^{\bar{\nabla}}(du^2)(X,Y,Z)
&= Y(du(X)du(Z)) - du(\bar{\nabla}_YX)du(Z) - du(X)du(\bar{\nabla}_YZ) \\
&\phantom{=} - Z(du(X)du(Y)) + du(\bar{\nabla}_ZX)du(Y) + du(X)du(\bar{\nabla}_ZY) \\[2mm]
&= (Ydu(X) - du(\bar{\nabla}_YX))du(Z) + du(X)(Ydu(Z)-du(\bar{\nabla}_YZ)) \\
&\phantom{=} - (Zdu(X) - du(\bar{\nabla}_ZX))du(Y) - du(X)(Zdu(Y)-du(\bar{\nabla}_ZY)).
\end{align*}
Each expression in parentheses is a Hessian, and using that the Hessian is a symmetric tensor lets us cancel the two terms that are equal and get 
\begin{equation}
\label{Derdu}
d^{\bar{\nabla}}(-2du^2)(X,Y,Z) = -2\mathrm{Hess}(X,Y)du(Z) +2\mathrm{Hess}(X,Z)du(Y).
\end{equation}
For the second term, writing out $d^{\bar{\nabla}}|\bar{\nabla}u|^2 \bar{g}$ and using the product rule reveals that 
\begin{align*}
d^{\bar{\nabla}}(|\bar{\nabla}u|^2 \bar{g})(X,Y,Z) 
&= Y(|\bar{\nabla}u|^2)\bar{g}(X,Z) - |\bar{\nabla}u|^2 (\bar{\nabla}_Z\bar{g})(X,Y) \\
&\phantom{=} - Z(|\bar{\nabla}u|^2)\bar{g}(X,Y) - |\bar{\nabla}u|^2(\bar{\nabla}_Y\bar{g})(X,Z).
\end{align*}
Because the connection is compatible with the metric, $\bar{\nabla}\bar{g} = 0$. 
Moreover, by writing $|\bar{\nabla}u|^2 = \bar{g}(\bar{\nabla} u, \bar{\nabla}u)$ we can see that the derivative of $|\bar{\nabla}u|^2$ is twice the Hessian.
So, 
\begin{equation}
\label{DerNorm}
d^{\bar{\nabla}}(|\bar{\nabla} u|^2\bar{g})(X,Y,Z)
= 2\mathrm{Hess}(\bar{\nabla}u,Y)\bar{g}(X,Z) - 2\mathrm{Hess}(\bar{\nabla}u,Z)\bar{g}(X,Y).
\end{equation}

By summing \eqref{Derdu} and \eqref{DerNorm}, we obtain the negative of \eqref{SimpDerHess}.
It only remains to show that the second term in \eqref{ToComp} vanishes. 
To this end, we will show $2du^2\owedge\bar{g} = |\bar{\nabla}u|^2\bar{g}\owedge \bar{g}$ when applied $(\bar{\nabla}u,X,Y,Z)$.
Indeed,
\begin{align*}
2du^2\owedge\bar{g}(\bar{\nabla}u,X,Y,Z)
&= 2du(\bar{\nabla}u)du(Z)\bar{g}(X,Y) + 2du(X)du(Y)\bar{g}(\bar{\nabla}u,Z) \\
&\phantom{=} - 2du(\bar{\nabla}u)du(Y)\bar{g}(X,Z) - 2du(X)du(Z)\bar{g}(\bar{\nabla}u,Y) \\[5pt]
&= 2|\bar{\nabla}u|^2\bar{g}(\bar{\nabla}u,Z)\bar{g}(X,Y) + 2du(X)du(Y)du(Z) \\
&\phantom{=} - 2|\bar{\nabla}u|^2\bar{g}(\bar{\nabla}u,Y)\bar{g}(X,Z) -2du(X)du(Y)du(Z).
\end{align*}
After cancelling the second and fourth terms, we recognize this as $|\bar{\nabla}u|^2\bar{g}\owedge\bar{g}$ evaluated at $(\bar{\nabla}u,X,Y,Z)$, as claimed.
Consequently, $(- 2 du^2 + |\bar{\nabla}u|^2 \bar{g})\owedge \bar{g} = 0$ on $(\bar{\nabla}u,X,Y,Z)$ and we have that \eqref{OSu} solves the Codazzi equation at infinity.


\end{proof}

\begin{thm}
\label{Uniqueness}
Let $\hat{g}$ be a locally conformally flat metric on $M$ which has dimension $n \geq 3$.
Then the solution $\hat{B}$ to the Gauss-Codazzi equations at infinity is unique.
\end{thm}

\begin{proof}
Assume $\hat{\two} = h$ and $\hat{\two} = k$ are two solutions to $\widehat{Rm} = (-1/2)\hat{g}\owedge\hat{\two}$.
Then we can write $\hat{g}\owedge (h-k) = 0$ and, using the trace identity \eqref{TraceKN} for the Kulkarni-Nomizu product, we have $(n-2)(h-k) + \tr_{\hat{g}}(h-k)\hat{g} = 0$.
Take a further trace of this equation to get $2(n-1)\tr_{\hat{g}}(h-k) = 0$, which implies $\tr_{\hat{g}}(h-k) = 0$.
Substituting this into the first trace identity gives $(n-2)(h-k) = 0$. 
As $n \neq 2$, this says $h=k$ and the solution is unique.
\end{proof}

Theorem \ref{Uniqueness} is not true in dimension $n = 2$ due to the distinction between complex structures and projective structures in this setting.
What is true is the following.
Fix a conformal metric $\hat{g}$ on the Riemann surface $M$.
Then for each solution $\hat{B}$ to the Gauss-Codazzi equations at infinty, there exists a unique projective structure such that $\hat{B} = 2\mathrm{OS}(\hat{g}) - K(\hat{g})\hat{g}$.
In particular, the solutions for $\hat{B}$ are in bijection with the set of projective structures compatible with $M$.
See \cite{Bridgeman-Bromberg2022} for the statements and proofs in this case.

The previous theorems give the following corollary in dimensions $n \geq 3$, which is Theorem \ref{bigthm1} from the introduction. 
\begin{cor}
\label{MainThm}
A pair $(\hat{g},\hat{B})$ solves the Gauss-Codazzi equations at infinity if and only if the metric $\hat{g}$ is locally conformally flat and
\[
\hat{g} \hat{B} = 2\mathrm{OS}(\hat{g}) - \frac{S(\hat{g})}{n(n-1)}\hat{g},
\]
where $\mathrm{OS}(\hat{g})$ is the Osgood-Stowe tensor of $\hat{g}$ with respect to the induced locally conformally flat structure.
\end{cor}

\section{The Weyl-Schouten Theorem}

Given a Riemannian metric $g$ on a smooth manifold $M$, one forms the Ricci tensor by taking the trace of the full curvature tensor $Ric(g) = \tr_g(Rm(g))$.
Here the trace may be thought of as a linear operator from $\mathcal{R}(T^*M)$, the sub-vector bundle of 4-tensors with the same symmetries as the curvature tensor, to $\Sigma^2(T^*M)$, the set of symmetric 2-tensors.
For $n\geq 3$, a right inverse of the trace $\tr_g: R(T^*M) \to \Sigma^2(T^*M)$ is given by 
\[
G(h) = \frac{1}{n-2}\left( h - \frac{\tr_g(h)}{2(n-1)} g \right) \owedge g,
\]
and the image of $G$ is the orthogonal complement to $\ker(\tr_g)$.

The Schouten tensor of $g$ is the symmetric 2-tensor $P(g)$ satisfying $G(Ric) = P(g) \owedge g$, i.e., 
\[
P(g) = \frac{1}{n-2}\left( Ric(g) - \frac{S(g)}{2(n-1)} g \right),
\]
and decomposing the curvature tensor via $\mathcal{R}(T^*M) = \ker(\tr_g) \oplus \ker(\tr_g)^\perp$ defines the Weyl tensor $W(g)$ by 
\begin{equation}
\label{Weyl}
Rm(g) = W(g) + P(g)\owedge g.
\end{equation}
One sees that 
\[
W(g) = Rm(g) - \frac{1}{n-2}Ric(g) \owedge g + \frac{S(g)}{2(n-1)(n-2)} g \owedge g.
\]
Note that neither the Weyl nor Schouten tensor exist in dimension $n = 2$.
\begin{lem}
\label{SchoutenSolves}
Suppose $g$ is a locally conformally flat metric on a manifold $M$ of dimension $n \geq 3$. 
Then 
\[
2\mathrm{OS}(g) - \frac{S(g)}{n(n-1)}g = -2 P(g). 
\]
In particular, $(g,-2P)$ solves the Gauss-Codazzi equations at infinity if and only if $\hat{g}$ is locally conformally flat. 
\end{lem}

\begin{proof}
The traceless Ricci tensor $Ric_0(g)$ for a metric $g$ locally conformal to a flat metric $\bar{g}$ satisfies $Ric_0(g) = -(n-2)\mathrm{OS}(g,\bar{g})$.
Substituting this into
\[
P(g) = \frac{1}{n-2}\left( Ric_0(g) + \frac{1}{n}S(g)g - \frac{S(g)}{2(n-1)} g \right)
\]
and simplifying gives the equality.
That $(g,-2P)$ solves \ref{GCInf} then follows from Corollary \ref{MainThm} and conversely if $(g,-2P)$ solves \ref{GCInf} then Proposition \ref{GCInf-LCF} shows $\hat{g}$ is locally conformally flat. 
\end{proof}

Because of this, one can interpret the solutions $2\mathrm{OS}(g) - K(g)g$ as a type of extension of the Schouten tensor to $n=2$ dimensions, for conformal metrics $g$ on Riemann surfaces.

The previous lemma can also be seen from the works of Espinar, G\'alvez, and Mira in \cite{Espinar-Galvez-Mira2009} where they consider immersed hypersurfaces induced by locally conformally flat metrics $\hat{g}$, and give the shape operator $B$ of the immersion in terms of the Schouten tensor of the metric $\hat{g}$. 
This can be inverted to give the shape operator at infinity $\hat{B}$ in terms of the Schouten tensor. 

\begin{thm}[Weyl-Schouten]
Let $g$ be a Riemannian metric on an $n$ dimensional manifold.

\begin{enumerate}
\item If $n = 3$ then $g$ is locally conformally flat if and only if $d^\nabla P = 0$.

\item If $n \geq 4$ then $g$ is locally conformally flat if and only if $W(g) = 0$.

\end{enumerate}
\end{thm}

\begin{proof}
Let $n \geq 3$. 
By Lemma \ref{SchoutenSolves}, the metric $g$ is locally conformally flat if and only if $(g,-2P)$ solves the Gauss-Codazzi equations at infinity.
This happens if and only if $d^{\nabla}P = 0$ and 
\[
Rm = -\frac{1}{2}g\owedge (-2P) = 0 + g \owedge P,
\]
which happens if and only if $d^\nabla P = 0$ and $W = 0$ by \eqref{Weyl}.

When $n = 3$, for dimension reasons, the Weyl tensor of any metric will always vanish (see, for example, \cite{Lee2018}).
So in dimension 3, the condition $d^\nabla P = 0$ is enough for the equivalence.
When $n \geq 4$, the identity
\[
d^\nabla P = -\frac{1}{n-3}\tr_g(\nabla W)
\]
shows that $d^\nabla P$ will vanish if $W=0$, so this condition is enough for the equivalence when $n\geq4$.
\end{proof}

\end{document}